\documentclass[12pt]{amsart}

\usepackage{amssymb}
\usepackage{mathtools}
\usepackage{bm}
\usepackage{graphicx}
\usepackage{psfrag}
\usepackage{color}
\usepackage{url}
\usepackage[linesnumbered, ruled, vlined]{algorithm2e}
\usepackage{algpseudocode}
\usepackage{fancyhdr}
\usepackage{xy}
\usepackage{dirtytalk}
\usepackage{float}
\usepackage{mathrsfs}
\input xy
\xyoption{all}

\usepackage{bm}

\usepackage{stmaryrd}

\usepackage{array}
\usepackage{longtable}
\usepackage[table]{xcolor}

\newtheorem{theorem}{Theorem}[section]

\newtheorem{prop}[theorem]{Proposition}
\newtheorem{lemma}[theorem]{Lemma}
\newtheorem{cor}[theorem]{Corollary}

\newtheorem{remark}[theorem]{Remark}
\theoremstyle{definition}
\newtheorem{definition}[theorem]{Definition}
\newtheorem{example}[theorem]{Example}
\numberwithin{equation}{section}

\newcommand\nn{\mathbb{N}}
\newcommand\qq{\mathbb{Q}}
\newcommand\rr{\mathbb{R}}

\keywords{well-ordered set, co-well-ordered set, positive monoid, Puiseux monoid, semiring, positive semiring, finite factorization property, factorization theory}
\subjclass[2010]{Primary: 20M13; Secondary: 16Y60, 06F05, 20M14}

\begin{document}
	
	\mbox{}
	\title{A Characterization of\\ Finite Factorization Positive Monoids}

	\author{Harold Polo}
	\address{Mathematics Department\\University of Florida\\Gainesville, FL 32611, USA}
	\email{haroldpolo@ufl.edu}
	\date{\today}
	
	\begin{abstract}
		We provide a characterization of the \emph{positive monoids} (i.e., additive submonoids of the nonnegative real numbers) that satisfy the finite factorization property. As a result, we establish that positive monoids with well-ordered generating sets satisfy the finite factorization property, while positive monoids with co-well-ordered generating sets satisfy this property if and only if they satisfy the bounded factorization property.
	\end{abstract}
	
	\maketitle
	
	\section{Introduction} \label{sec:intro}
	\medskip
	During their study of factorizations in integral domains, Anderson, Anderson, and Zafrullah~\cite{AAZ1} introduced the notion of \emph{finite factorization domains} (or FFDs), which are domains in which every nonzero element has finitely many non-associated divisors; alternatively, it is said that these domains satisfy the finite factorization property. The class of FFDs encompasses, most significantly, Krull domains, and FFDs have been extensively investigated (see, for instance, \cite{AM1996,GOTTI19}). As it was pointed out by Halter-Koch~\cite{halter-koch}, it is possible to study factorizations in the more general context of cancellative and commutative monoids, and many of the factorization properties introduced for integral domains have a monoid analog. In particular, a monoid $M$ is called a \emph{finite factorization monoid} (or an FFM) provided that every nonzero element of $M$ has finitely many non-associated divisors.
	\smallskip
	
	\emph{Positive monoids}, that is, additive submonoids of $\rr_{\geq 0}$, have played an important role in factorization theory. For example, Grams~\cite{Grams} used Puiseux monoids (i.e., additive submonoids of $\qq_{\geq 0}$) to refute Cohn's assertion~(\cite[Proposition 1.1]{cohn}) that every atomic integral domain satisfies the ascending chain condition on principal ideals. More recently, Bras-Amor\'{o}s~\cite{bras} highlighted connections between positive monoids and music theory, while Coykendall and Gotti~\cite{JCFG2019} employed Puiseux monoids to tackle a question posed by Gilmer almost four decades ago in \cite[page~$189$]{Gilmer}. The aim of the present article is to study the positive monoids that satisfy the finite factorization property. These monoids have been studied before; while Gotti~\cite{GOTTI19} showed that increasing positive monoids are FFMs, Baeth et al.~\cite{NBSCFG} investigated the dyadic notion of bi-FFSs in the context of \emph{positive semirings} (i.e., positive monoids that are closed under multiplication and contain the multiplicative identity). On the other hand, Correa-Morris and Gotti~\cite{JCFG} proved that the finite factorization property and the bounded factorization property coincide for positive semirings generated (as a monoid) by the nonnegative powers of a single element. 
	
	\smallskip
	This paper is structured as follows. We begin next section by introducing not only the necessary background but also the notation we shall be using throughout this manuscript. Then, in Section~3, we provide a characterization of finite factorization positive monoids. As a result, we establish that positive monoids with well-ordered generating sets satisfy the finite factorization property, while positive monoids with co-well-ordered generating sets satisfy this property if and only if they satisfy the bounded factorization property. We conclude by showing, in Section~$4$, that for certain positive semirings, the additive structure completely determines whether the multiplicative structure satisfies the finite factorization property.
	\bigskip
	
	\section{Background}
	
	We now review some of the standard concepts we shall be using later. The monograph~\cite{GH06b} by Geroldinger and Halter-Koch offers extensive background on non-unique factorization theory. 
	
	\subsection{Notation}
	
	Let $\nn$ denote the set of nonnegative integers, and let $\mathbb{P}$ denote the set of prime numbers. If $X$ is a subset of the real numbers then we set $X_{<r} \coloneqq \{x \in X \mid 0 \leq x < r\}$; we define $X_{\leq r}$, $X_{>r}$, and $X_{\geq r}$ in a similar way. Additionally, a subset $X$ of $\rr_{\geq 0}$ is called \emph{well-ordered} \footnote{ Usually, a subset $X \subseteq \rr_{\geq 0}$ is called well-ordered provided that every nonempty subset of $X$ has a minimal element, but assuming the Axiom of Choice this is equivalent to our definition.} provided that $X$ contains no infinite decreasing sequence; if $X$ contains no infinite increasing sequence then it is called \emph{co-well-ordered}. For a positive rational number $q = n/d$ with $n$ and $d$ relatively prime positive integers, we call $n$ the \emph{numerator} and $d$ the \emph{denominator} of $q$, and we set $\mathsf{n}(q) := n$ and $\mathsf{d}(q) := d$. For nonnegative integers $k$ and $m$, we denote by $\llbracket k,m \rrbracket$ the set of integers between $k$ and $m$, i.e., $\llbracket k,m \rrbracket \coloneqq \left\{ s \in \nn \mid k \leq s \leq m \right\}$.

	\subsection{Commutative Monoids}
	
	Throughout this paper, a \emph{monoid} is defined to be a semigroup with identity that is cancellative, commutative, and \emph{reduced} (i.e., its only invertible element is the identity), and we use additive notation for monoids. Let $M$ be a monoid. We denote by $\mathcal{A}(M)$ the set consisting of elements $a \in M^{\bullet} \coloneqq M\setminus\{0\}$ satisfying that if $a = x + y$ for some $x,y \in M$ then either $x = 0$ or $y = 0$; the elements of this set are called \emph{atoms}. For a subset $S \subseteq M$, we denote by $\langle S \rangle$ the smallest submonoid of $M$ containing $S$, and if $M = \langle S \rangle$ then it is said that $S$ is a \emph{generating set} of $M$. A monoid $M$ is \emph{atomic} if $M = \langle\mathcal{A}(M)\rangle$. For $x,y \in M$, it is said that $x$ \emph{divides} $y$ if there exists $x' \in M$ such that $y = x + x'$\! in which case we write $x \,|_M \,y$ and drop the subscript whenever $M = (\nn^{\bullet}, \times)$. We denote by $\mathsf{D}_M(x)$ the set of nonzero divisors of an element $x$ in $M$, and set $\mathsf{A}_M(x) \coloneqq \mathsf{D}_M(x) \cap \mathcal{A}(M)$; we omit subscripts whenever the monoid is clear from the context. A subset $I$ of $M$ is an \emph{ideal} of $M$ provided that $I + M \subseteq I$. An ideal $I$ is \emph{principal} if $I = x + M$ for some $x \in M$. Furthermore, it is said that $M$ satisfies the \emph{ascending chain condition on principal ideals} (or \emph{ACCP}) if every increasing sequence of principal ideals of $M$ eventually stabilizes. If $M$ satisfies the ACCP then it is atomic (\cite[Proposition 1.1.4]{GH06b}).
	
	Following \cite{NBSCFG}, we call additive submonoids of $\rr_{\geq 0}$ \emph{positive monoids}; if they are  submonoids of $\qq_{\geq 0}$ then we call them \emph{Puiseux monoids}. Since Puiseux monoids are the torsion-free rank-$1$ monoids that are not groups (\cite[Theorem 3.12]{GeGoTr2019}), they are, up to isomorphism, the positive monoids of rank $1$. The atomic structure of Puiseux monoids is convoluted and has received considerable attention lately (see \cite{ScGG220} and references therein). The most investigated subclass of Puiseux monoids is that one comprising all numerical monoids, i.e., additive submonoids of $\nn$ whose complement (in $\nn$) is finite. An introduction to numerical monoids can be found in~\cite{GSJCR2009}.
	
	\subsection{Factorizations}
	
	For the rest of the section, let $M$ be an atomic monoid. The \emph{factorization monoid} of $M$, denoted by $\mathsf{Z}(M)$, is the free (commutative) monoid on $\mathcal{A}(M)$. The elements of $\mathsf{Z}(M)$ are called \emph{factorizations}, and if $z = a_1 + \cdots + a_n \in \mathsf{Z}(M)$ for $a_1, \ldots, a_n \in\mathcal{A}(M)$ then it is said that the \emph{length} of $z$, denoted by $|z|$, is $n$. We assume that the empty factorization has length $0$. The unique monoid homomorphism $\pi\colon\mathsf{Z}(M) \to M$ satisfying that $\pi(a) = a$ for all $a \in\mathcal{A}(M)$ is called the \emph{factorization homomorphism} of $M$. For each $x \in M$, there are two important sets associated to $x$:
	\[
	\mathsf{Z}_M(x) \coloneqq \pi^{-1}(x) \subseteq \mathsf{Z}(M) \hspace{0.6 cm}\text{ and } \hspace{0.6 cm}\mathsf{L}_M(x) \coloneqq \left\{|z| : z \in\mathsf{Z}_M(x)\right\},
	\]
	which are called the \emph{set of factorizations} of $x$ and the \emph{set of lengths} of $x$, respectively; as usual we drop the subscript whenever the monoid is clear from the context. Additionally, the collection $\mathcal{L}(M) \coloneqq \{\mathsf{L}(x) \mid x \in M\}$ is called the \emph{system of sets of lengths} of $M$. See~\cite{aG16} for a survey on sets of lengths and the role they play in factorization theory. It is said that $M$ is a \emph{finite factorization monoid} (or an FFM) if $\mathsf{Z}(x)$ is nonempty and finite for all $x \in M$. Similarly, it is said that $M$ is a \emph{bounded factorization monoid} (or BFM) if $\mathsf{L}(x)$ is nonempty and finite for all $x \in M$. Clearly, an FFM is also a BFM, while a BFM satisfies the ACCP by \cite[Corollary~1.3.3]{GH06b}.
	\smallskip
	
	\section{Positive Monoids}
	
	In this section, we provide a characterization of the positive monoids that satisfy the finite factorization property. As a result, we obtain not only that positive monoids with well-ordered generating sets are FFMs, but also that positive monoids with co-well-ordered generating sets are FFMs if and only if they are BFMs. But first we need to collect a lemma, which is a generalization of \cite[Theorem~4.7]{ScGG220}.

\begin{lemma} \label{sufficient condition for BFM}
	Let $P$ be a positive monoid. The monoid $P$ is a BFM provided that $\inf \mathsf{D}(x) > 0$ for every $x \in P^{\bullet}$.
\end{lemma}

\begin{proof}
	Take an arbitrary element $x \in P^{\bullet}$. There exists $\varepsilon \in \rr_{>0}$ such that $\varepsilon < \inf \mathsf{D}(x)$. Clearly, the element $x$ can be written as the sum of at most $\lceil x/\varepsilon \rceil$ elements of $P^{\bullet}$. Now let $x = a_1 + \cdots + a_n$, where $a_1, \ldots, a_n \in P^{\bullet}$, and assume without loss of generality that $n$ is maximal. Then it is not hard to see that $a_i \in \mathcal{A}(P)$ for each $i \in \llbracket 1,n \rrbracket$. Since $x$ was arbitrarily taken, the monoid $P$ is atomic. Moreover, for each $x \in P$, we have that $|z| < \lceil x/\varepsilon \rceil$ for every $z \in \mathsf{Z}(x)$. Therefore, $P$ is a BFM.
\end{proof}

\begin{cor}\cite[Theorem~4.7]{ScGG220} \label{cor: if 0 not a limit point then it is a BFM}
	Let $P$ be a positive monoid. If $0$ is not a limit point of $P^{\bullet}$ then $P$ is a BFM.
\end{cor}

\begin{proof}
	Since $0$ is not a limit point of $P^{\bullet}$, we have that the inequality $\inf\mathsf{D}(x) > 0$ holds for every $x \in P^{\bullet}$, and the result follows from Lemma~\ref{sufficient condition for BFM}.
\end{proof}

Now we are in a position to prove the main result of this section.

\begin{theorem} \label{theorem: characterization of positive monoids that are FFMs}
	Let $P$ be a positive monoid. Then $P$ is an FFM if and only if there is no $x \in P$ such that $x$ is a limit point of $\mathsf{D}(2x)$.
\end{theorem}

\begin{proof}
	If there exists $x \in P$ such that $x$ is a limit point of $\mathsf{D}(2x)$ then the element $2x \in P$ has infinitely many (non-associated) divisors in $P$, and the direct implication follows from \cite[Proposition 1.5.5]{GH06b}.
	
	To tackle the reverse implication, we first prove that $P$ is a BFM. Suppose, towards a contradiction, that there exists $x \in P^{\bullet}$ such that $0$ is a limit point of $\mathsf{D}(x)$. Then there exists a strictly decreasing sequence $(d_n)_{n \in \nn}$ of elements of $\mathsf{D}(x)$ converging to $0$, which implies that $\{x - d_n, x + d_n\} \subseteq P$ for every $n \in \nn$. Consequently, $x$ is a limit point of $\mathsf{D}(2x)$. This contradiction proves that our hypothesis is untenable. So for every $x \in P$ we have that $\inf \mathsf{D}(x) > 0$ which, in turn, implies that $P$ is a BFM by Lemma~\ref{sufficient condition for BFM}.
	
	Now assume that $P$ is not an FFM. By \cite[Proposition 1.5.5]{GH06b}, there exists $x \in P^{\bullet}$ such that the set $\mathsf{A}(x)$ has infinite cardinality. Since $\mathsf{L}(x)$ is finite, there exists $l \in \mathsf{L}(x)$ such that the set $Z = \{z \in\mathsf{Z}(x) : |z| = l\}$ has infinite cardinality too. Let us denote by $A_*$ the set consisting of the atoms of $P$ that show up in, at least, one factorization in $Z$. Clearly, we have $|A_*| = \infty$. Next we describe a procedure to obtain a sequence of factorizations $(z_n = a_{n}^{1} + \cdots + a_{n}^l)_{n \in \nn}$ such that $z_n \in Z$ for each $n \in \nn$ and, for each $i \in \llbracket 1,l \rrbracket$, the sequence $(a_n^i)_{n \in \nn}$ is constant, strictly increasing, or strictly decreasing. For each $i \in \llbracket 1,l \rrbracket$, let us denote by $A_i$ the set formed by the $i$th smallest atoms of the factorizations in $Z$. Since $|A_*| = \infty$, there exists $j \in \llbracket 1,l \rrbracket$ such that $|A_j| = \infty$. There is no loss in assuming that $j$ is minimal. Since $A_j$ is an infinite bounded subset of the nonnegative real numbers, it contains a sequence that is either strictly increasing or strictly decreasing. Consequently, there exists a sequence $(S_1)$ of elements of $Z$ such that the sequence induced by $(S_1)$ in $A_j$ is either strictly increasing or strictly decreasing. Since $|A_i| < \infty$ for each $i \in \llbracket 1,j - 1 \rrbracket$, there is no loss in assuming that the sequence induced by $(S_1)$ in $A_i$ is constant for each $i \in \llbracket 1,j - 1 \rrbracket$. More generally, if $i \in \llbracket 1,l \rrbracket$ and $|A_i| < \infty$ then we may assume that the sequence induced by $(S_1)$ in $A_i$ is constant. For each $i \in \llbracket 1,l \rrbracket$, let us denote by $(A_i)^1$ the sequence induced by $(S_1)$ in $A_i$. Assume that we already defined, for some $j \in \nn^{\bullet}$, a sequence $(S_j)$ of elements of $Z$. If each sequence $(A_i)^j$ (with $i \in \llbracket 1,l \rrbracket$) is constant, strictly increasing, or strictly decreasing then our procedure stops. Otherwise, there exists $k \in \llbracket 1,l \rrbracket$ such that the sequence $(A_k)^j$ has infinitely many distinct elements and is neither strictly increasing nor strictly decreasing. Once again, assume that $k$ is minimal. Clearly, the inequality $j < k$ holds. Since the underlying set of $(A_k)^j$ is infinite and bounded, there exists an infinite subsequence $(S_{j + 1})$ of $(S_j)$ such that the sequence induced by $(S_{j + 1})$ in $(A_k)^j$ is either strictly increasing or strictly decreasing. For each $i \in \llbracket 1,l \rrbracket$, let $(A_i)^{j + 1}$ be the sequence induced by $(S_{j + 1})$ in $(A_i)^j$. Since $(S_{j + 1})$ is a subsequence of $(S_j)$, we have that $(A_i)^{j + 1}$ is a subsequence of $(A_i)^j$ for each $i \in \llbracket 1,l \rrbracket$. By induction, it follows that there exists a sequence of factorizations $\sigma = (z_n = a_{n}^{1} + \cdots + a_{n}^l)_{n \in \nn}$ such that $z_n \in Z$ for each $n \in \nn$ and, for each $i \in \llbracket 1,l \rrbracket$, the sequence $(a_n^i)_{n \in \nn}$ is constant, strictly increasing, or strictly decreasing. 
	
	We already established that there exists $j \in \llbracket 1,l \rrbracket$ such that the sequence $(a_n^j)_{n \in \nn}$ is either strictly increasing or strictly decreasing. Furthermore, there is no loss in assuming that none of the sequences $(a_n^i)_{n \in \nn}$ is constant; otherwise, we can just take the subfactorizations of the elements of $\sigma$ that do not include these atoms. As a consequence, there exist $k,r \in \llbracket 1,l \rrbracket$ such that $(a_n^k)_{n \in \nn}$ is strictly increasing and $(a_n^r)_{n \in \nn}$ is strictly decreasing. Indeed, if for example all sequences $(a_n^i)_{n \in \nn}$ are strictly increasing then there exist two factorizations $z,z' \in \mathsf{Z}(x)$ such that $\pi(z) > \pi(z')$, which is impossible. Suppose, without loss of generality, that there exists $t \in \llbracket 1, l - 1 \rrbracket$ such that the sequence $(a_n^i)_{n \in \nn}$ is strictly increasing for every $i \in \llbracket 1,t \rrbracket$, while the sequence $(a_n^j)_{n \in \nn}$ is strictly decreasing for each $j \in \llbracket t + 1, l \rrbracket$. For each $i \in \llbracket 1,l \rrbracket$, set $l_i \coloneqq \lim_{n \to \infty} a_n^i$. Let $\varepsilon \in \rr_{>0}$ such that $\varepsilon < x$. Now fix $N \in \nn$ such that $|l_i - a_N^i| < \varepsilon/l$ for every $i \in \llbracket 1,l \rrbracket$. As the reader can easily verify, the following equalities hold
	\[
		x = \sum_{i = 1}^{t} a_{N + 1}^i + \sum_{i = t + 1}^{l} a_N^i - \sum_{i = 1}^t (a_{N + 1}^i - a_N^i) = \sum_{i = 1}^{t} a_{N}^i + \sum_{i = t + 1}^{l} a_{N + 1}^i + \sum_{i = t + 1}^{l} (a_N^i - a_{N + 1}^i).
	\]
	Let $\delta = \sum_{i = 1}^t (a_{N + 1}^i - a_N^i)$, and note that $0 < \delta < \varepsilon$. Since $\sum_{i = 1}^{t} (a_{n + 1}^i - a_{n}^i) = \sum_{i = t + 1}^{l} (a_n^i - a_{n + 1}^i)$ for each $n \in \nn$, we have that $x - \delta$ and $x + \delta$ are both elements of $P$. Hence $x$ is a limit point of $\mathsf{D}(2x)$, from which our result follows.
\end{proof}

\begin{cor} \label{cor: well-ordered positive monoids}
	Let $P$ be a positive monoid with a well-ordered generating set. Then $P$ is an FFM.
\end{cor}

\begin{proof}
	Since $P$ has a well-ordered generating set, the set $P$ is also well-ordered by \cite[Theorem~3.4]{neumann} and, consequently, there is no $x \in P$ such that $x$ is a limit point of $\mathsf{D}(2x)$.
\end{proof}

\begin{remark}
	{\normalfont The definition of well-ordered sets used by Neumann~\cite{neumann} is different from ours. However, these two definitions are equivalent as the author pointed out in \cite[Lemma~3.1]{neumann}.}
\end{remark}

\begin{remark}
	{\normalfont Notice that Corollary~\ref{cor: well-ordered positive monoids} is a generalization of \cite[Theorem~4.19]{ScGG220}, which states that increasing Puiseux monoids are FFMs. Also note that Corollary~\ref{cor: well-ordered positive monoids} can be proved independently of Theorem~\ref{theorem: characterization of positive monoids that are FFMs}. In fact, by Corollary~\ref{cor: if 0 not a limit point then it is a BFM}, if $P$ is a positive monoid with a well-ordered generating set then $0$ is not a limit point of $P^{\bullet}$, which implies that $P$ is a BFM and, thus, atomic. If for some $x \in P$ the set $\mathsf{A}(x)$ has infinite cardinality then it is not hard to construct a strictly decreasing sequence of elements of $P$, which contradicts \cite[Theorem~3.4]{neumann}.}
\end{remark}

\begin{cor} \label{cor: BFM and not increasing sequence of atoms imply FFM}
	Let $P$ be a positive monoid with a co-well-ordered generating set. Then $P$ is an FFM if and only if $P$ is a BFM.
\end{cor}

\begin{proof}
	 The direct implication trivially follows. As for the remaining implication, suppose by way of contradiction that $P$ is not an FFM. In the proof of Theorem~\ref{theorem: characterization of positive monoids that are FFMs}, we established that in this case $\mathcal{A}(P)$ contains at least one increasing sequence. Since $P$ is atomic (and reduced), we have that $\mathcal{A}(P) \subseteq S$ for any generating set $S$ of $P$. Consequently, no generating set of $P$ is co-well-ordered, a contradiction.
\end{proof}

As the following example illustrates, not all positive monoids satisfying the finite factorization property have either well-ordered or co-well-ordered generating sets. In particular, the converse of Corollary~\ref{cor: well-ordered positive monoids} does not hold.

\begin{example} \label{example: sufficient conditions for Puiseux monoids to satisfy the finite factorization property}
	For each $n \in \nn^{\bullet}$, let $p_n$ denote the $n$th prime number, and consider the Puiseux monoid $M$ generated by the set $S = \left\{3 + 1/p_{2n}, 3 - 1/p_{2n + 1} \mid n \in \nn^{\bullet}\right\}$. It is easy to show that $\mathcal{A}(M) = S$, which implies that $M$ is atomic. Since $0$ is not a limit point of $M^{\bullet}$, the monoid $M$ is a BFM by Corollary~\ref{cor: if 0 not a limit point then it is a BFM}. Furthermore, $M$ is an FFM. Indeed, for $x \in M$ and $a \in \mathcal{A}(M)$, it is not hard to show that if $a \,|_{M}\, x$ then either $\mathsf{d}(a) \,|\, \mathsf{d}(x)$ or $3\cdot\mathsf{d}(a) \,|_{M}\, x$, which implies that $x$ has finitely many divisors in $M$. However, $\mathcal{A}(M)$ is neither well-ordered nor co-well-ordered.
\end{example}

\begin{cor}
	Let $M = \langle S \rangle$ be a Puiseux monoid satisfying that $0$ is not a limit point of $M^{\bullet}$ and $\gcd(\mathsf{d}(s),\mathsf{d}(s')) = 1$ for $s$ and $s'$ distinct elements of $S$. Then $M$ is an FFM.
\end{cor}

\begin{proof}
	By Corollary~\ref{cor: if 0 not a limit point then it is a BFM}, the monoid $M$ is a BFM. On the other hand, it is not hard to check that $\mathcal{A}(M) = S$. Suppose towards a contradiction that $M$ is not an FFM. As part of the proof of Theorem~\ref{theorem: characterization of positive monoids that are FFMs}, we established that if $M$ is a positive BFM that is not an FFM then for all $\varepsilon \in \rr_{>0}$ there exist increasing sequences $(a_n^1)_{n \in \nn}, \ldots, (a_n^k)_{n \in \nn}$ and decreasing sequences $(b_n^1)_{n \in \nn}, \ldots, (b_n^t)_{n \in \nn}$ of atoms of $M$ such that $\sum_{i = 1}^k (a_{n + 1}^i - a_n^i) = \sum_{i = 1}^{t} (b_n^i - b_{n + 1}^i) < \varepsilon$ for all $n \in \nn$. Assume, without loss of generality, that the underlying sets of the sequences $(a_n^i)_{n \in \nn}$ and $(b_n^j)_{n \in \nn}$ are disjoint for $i \in \llbracket 1,k \rrbracket$ and $j \in \llbracket 1,t \rrbracket$. From this observation, it is not hard to show that our previous equation does not hold for any $\varepsilon$ strictly less than $1$, which is a contradiction.
\end{proof}

\subsection{Submonoids of Finite Factorization Positive Monoids}

It is well known that a submonoid of a reduced FFM is an FFM (\cite[Corollary 1.5.7]{GH06b}). However, the finite factorization property does not ascend from a submonoid to the monoid (and the reader can easily verify this using Theorem~\ref{theorem: characterization of positive monoids that are FFMs}). Next we show that a positive monoid $P$ satisfies the finite factorization property if and only if certain submonoids of $P$ satisfy it, but first let us introduce a definition.

\begin{definition}
	Given a subset $S \subseteq \rr_{\geq 0}$, we denote by $\mathfrak{l}(S)$ the set of limit points of $S$ contained in $S$. 
\end{definition}

\begin{prop} \label{prop: submonoids of FFMs}
	Let $\langle S \rangle$ be a positive monoid, and let $A \subseteq S$ be closed in $\rr_{\geq 0}$ such that $\mathfrak{l}(S) = \mathfrak{l}(S \setminus A)$. Then $\langle S \rangle$ is an FFM if and only if $\langle S \setminus A \rangle$ is an FFM.
\end{prop}

\begin{proof}
	Set $P \coloneqq \langle S \rangle$ and $P' \coloneqq \langle S \setminus A \rangle$. To tackle the nontrivial implication, assume by way of contradiction that $P$ is not an FFM. By Theorem~\ref{theorem: characterization of positive monoids that are FFMs}, there exists $x \in P$ such that for every $n \in \nn^{\bullet}$ there exists $0 < \delta_n < 1/n$ satisfying that $\{x - \delta_n, x + \delta_n\} \subseteq P$. Since $P'$ is an FFM, the set $B = \{x - \delta_n, x + \delta_n \mid n \in \nn^{\bullet}\} \setminus P'$ has infinite cardinality; otherwise, the element $2x \in P'$ would have infinitely many (non-associated) divisors. It is easy to see that each element of $B$ is divisible in $P$ by some element of $A$. Let us denote by $A'$ the set consisting of the elements of $A$ that divide some element in $B$. We claim that $|A'| < \infty$. In fact, if $A'$ is an infinite subset of $A$ then there exists $l \in \rr_{\geq 0}$ such that $l$ is a limit point of $A'$ by Bolzano-Weierstrass Theorem, which states that each bounded sequence in $\rr$ has a convergent subsequence. Since $A$ is a closed subset of $\rr_{\geq 0}$, we have that $l \in A$, but this contradicts the equality $\mathfrak{l}(S) = \mathfrak{l}(S \setminus A)$, and our claim follows. Now let $D = \{x - \delta_n \mid n \in \nn^{\bullet}\}$ and $C = \{x + \delta_n \mid n \in \nn^{\bullet}\}$. If the set $D \cap P'$ has infinitely many elements then set $a_1 \coloneqq 0$; otherwise, take $a_1$ to be the maximal element of $\langle A' \rangle$ dividing (in $P$) infinitely many elements of $D$. After replacing $(\delta_n)_{n \in \nn}$ by a suitable subsequence $(\alpha_n)_{n \in \nn}$, we have that $a_1$ divides in $P$ all elements of $D$ and $\{x - a_1 - \alpha_n \mid n \in \nn^{\bullet}\} \subseteq P'$. Similarly, there is no loss in assuming that there exists $a_2 \in \rr_{\geq 0}$ such that $a_2$ divides in $P$ all elements of $C$ and $\{x - a_2 + \alpha_n \mid n \in \nn^{\bullet}\} \subseteq P'$. Consequently, the element $2x - a_1 - a_2 \in P'$ has infinitely many (non-associated) divisors in $P'$, which contradicts \cite[Proposition 1.5.5]{GH06b}. 
\end{proof}

Following \cite{GG17}, we say that a sequence of real numbers is \emph{strongly increasing} if it increases to infinity.

\begin{cor} \label{cor: taking off a strongly increasing sequence of generators does not change the FFP}
	Let $\langle S \rangle$ be a positive monoid, and let $A \subseteq S$ be the underlying set of a strongly increasing sequence. Then $\langle S \rangle$ is an FFM if and only if $\langle S \setminus A \rangle$ is an FFM.
\end{cor}

The atomicity of \emph{rational multicyclic monoids}, that is, additive submonoids of the nonnegative rational numbers generated by multiple geometric sequences, was briefly studied in \cite{HP2020}. Next we show that, in this context, the finite factorization property only depends on the generators with values strictly less than $1$.

\begin{cor} \label{cor: the finite factorization property in multicyclic monoids}
	Let $\mathcal{B}$ be a finite subset of $\rr_{>0}$, and set $M_{\mathcal{B}} \coloneqq \langle b^n \mid b \in \mathcal{B}, \, n \in \nn \rangle$. The following statements hold.
	\begin{enumerate}
		\item If $\mathcal{B}' = \mathcal{B} \cap (0,1)$ then $M_{\mathcal{B}}$ is an FFM if and only if $M_{\mathcal{B}'}$ is an FFM.
		\smallskip
		\item If $b \geq 1$ for each $b \in \mathcal{B}$ then $M_{\mathcal{B}}$ is an FFM.
	\end{enumerate}
\end{cor}

\begin{proof}
	It immediately follows from Corollary~\ref{cor: taking off a strongly increasing sequence of generators does not change the FFP}.
\end{proof}
\smallskip

\section{Positive Semirings with Finitely Many Bi-atoms}

\emph{Positive semirings}, that is, positive monoids that are closed under multiplication and contain the multiplicative identity, have received considerable attention lately. For example, in \cite{JCFG} the authors studied the atomic properties of the additive structure of positive algebraic valuations of $\nn[X]$, the semiring of polynomials with nonnegative coefficients, while some of the factorization invariants of $\nn[\tau]$, where $\tau$ is a quadratic integer, were investigated in \cite{CCMS09}. Most relevant to the work on this section, Baeth et al.~\cite{NBSCFG} investigated the dualistic nature of the finite factorization property in the context of positive semirings.

\begin{definition}
	Following \cite{NBSCFG}, we say that a positive semiring $(S, +, \cdot)$ is a \emph{bi-FFS} if both $(S,+)$ and $(S^{\bullet}, \cdot)$ are FFMs. In a similar manner, we use the terminologies \emph{bi-BFS}, \emph{bi-ACCP}, \emph{bi-atomic}, and \emph{bi-reduced}. Additionally, we say that an element $a \in S$ is a \emph{bi-atom} if it is an atom of $(S,+)$ and $(S^{\bullet},\cdot)$.
\end{definition}

Studying the finite factorization property in the context of all positive semirings is beyong the scope of this paper. Here we only consider positive semirings that are bi-atomic and bi-reduced, and contain finitely many bi-atoms. We restrict ourselves to this subclass because, as we now show, in this case we can ignore the multiplicative structure. 

\begin{definition}
	Given a bi-atomic positive semiring $(S,+,\cdot)$, we denote by $\mathcal{A}_{+}(S)$ and $\mathcal{A}_{\times}(S^{\bullet})$ the set of atoms of $(S,+)$ and $(S^{\bullet},\cdot)$, respectively.
\end{definition}

\begin{prop} \label{theorem: finitely generated bi-FFS}
	Let $S$ be a bi-atomic and bi-reduced positive semiring satisfying that $|\mathcal{A}_+(S) \cap \mathcal{A}_{\times}(S^{\bullet})| < \infty$. The following statements are equivalent. 
	
	\begin{enumerate}
		\item $(S,+, \cdot)$ is a bi-FFS.
		\smallskip
		\item $(S,+)$ is an FFM.
		\smallskip
		\item There is no $x \in S$ such that $x$ is a limit point of $\mathsf{D}_{(S,+)}(2x)$.
	\end{enumerate}
\end{prop}

\begin{proof}
	By Theorem~\ref{theorem: characterization of positive monoids that are FFMs}, the statements $(2)$ and $(3)$ are equivalent. On the other hand, proving that $(1)$ and $(2)$ are equivalent reduces to show that $(2)$ implies $(1)$. Assume towards a contradiction that $(S^{\bullet},\cdot)$ is not an FFM. Then there exists $s_0 \in S^{\bullet}$ such that $|\mathsf{Z}_{(S^{\bullet}, \cdot)}(s_0)| = \infty$. Since the inequality $|\mathcal{A}_+(S) \cap\, \mathcal{A}_{\times}(S^{\bullet})| < \infty$ holds, the set $A = \{a \in \mathcal{A}_{\times}(S^{\bullet}) \setminus \mathcal{A}_+(S) : a \,|_{(S^{\bullet},\cdot)}\, s_0 \}$ has infinite cardinality. Clearly, for each $a \in A$ there exist $x_a, y_a \in S^{\bullet}$ such that $a = x_a + y_a$ which, in turn, implies that for each $a \in A$ there exists $k_a \in S^{\bullet}$ such that $s_0 = k_ax_a + k_ay_a$. Since $(S,+)$ is a reduced positive FFM, there exists an infinite subset $A'$ of $A$ satisfying that $k_ax_a = k_bx_b$ for all $a,b \in A'$; otherwise, the element $s_0$ would have infinitely many additive (non-associated) divisors in $(S,+)$, which is a contradiction. Consequently, we also have that the equality $k_ay_a = k_by_b$ holds for all $a,b \in A'$. Since $A'$ has infinite cardinality, either $\{x_a \mid a \in A'\}$ or $\{y_a \mid a \in A'\}$ has infinite cardinality. Without loss of generality, assume that $\{x_a \mid a \in A'\}$ has infinite cardinality, and fix $a \in A'$. Then set $s_1 \coloneqq k_ax_a \in S^{\bullet}$. By \cite[Proposition 1.5.5]{GH06b}, we have $|\mathsf{Z}_{(S^{\bullet},\cdot)}(s_1)| = \infty$. Evidently, we can recursively apply this idea to generate an infinite sequence $s_0, s_1, \ldots$ of elements of $S^{\bullet}$ such that $s_{i + 1} \,|_{(S,+)}\, s_i$ for each $i \in \nn$. But this contradicts that $(S,+)$ satisfies the ACCP, and our argument concludes.
\end{proof}

\begin{cor} \label{cor: characterizations of FFM real cyclic semirings}
	Let $r \in \rr_{>0}$ such that $\nn[r]$ is bi-atomic. Then the following statements are equivalent.

	\begin{enumerate}
		\item $(\nn[r],+)$ is an FFM.
		\smallskip
		\item $(\nn[r],+)$ is a BFM.
		\smallskip
		\item $(\nn[r],+)$ satisfies the ACCP.
		\smallskip
		\item $(\nn[r],+,\cdot)$ is a bi-FFS.
		\smallskip
		\item $(\nn[r],+,\cdot)$ is a bi-BFS.
		\smallskip
		\item $(\nn[r],+,\cdot)$ satisfies the bi-ACCP.
		\smallskip
	\end{enumerate}
\end{cor}

\begin{proof}
	By \cite[Proposition 3.2]{NBSCFG}, the semiring $\nn[r]$ is bi-reduced. Moreover, it is easy to see that $\mathcal{A}_+(\nn[r]) \cap \mathcal{A}_{\times}(\nn[r]^{\bullet}) = \{r\}$. The first four statements are equivalent by \cite[Theorem 4.11]{JCFG} and Proposition~\ref{theorem: finitely generated bi-FFS}. Note that, starting at $(4)$, each statement implies the next one and, clearly, $(6)$ implies $(3)$.
\end{proof}

Not all bi-atomic and bi-reduced positive semirings containing finitely many bi-atoms are bi-FFSs. Consider the following example.

\begin{example}
	Let $q \in \qq_{<1}$ such that $\mathsf{n}(q) > 1$ and $\mathsf{d}(q) \in \mathbb{P}$, and consider the positive semiring $\nn[q]$. By \cite[Proposition 4.3]{NBSCFG}, $\nn[q]$ is bi-atomic but does not satisfy the bi-ACCP, so in particular it is not a bi-FFS. Note that $\nn[q]$ is bi-reduced by virtue of \cite[Proposition 3.2]{NBSCFG}.
\end{example}

Unfortunately, Proposition~\ref{theorem: finitely generated bi-FFS} cannot be extended to the more general class of bi-atomic and bi-reduced positive semirings as the following example (which is a construction introduced in \cite{NBSCFG}) illustrates.

\begin{example} \label{example of a non-finitely generated bi-atomic and bi-reduced positive semiring that is not bi-FFS}
	Let $P$ be an infinite subset of $\mathbb{P}$, and let $M = \langle 1/p \mid p \in P \rangle$. Let us consider the positive semiring $E(M) \coloneqq \langle e^m \mid m \in M \rangle$. The additive monoid $E(M)$ is free on the set $\{e^m \mid m \in M\}$ by Lindemann-Weierstrass Theorem stating that, for distinct algebraic numbers $\alpha_1, \ldots, \alpha_n$, the set $\{e^{\alpha_1}, \ldots, e^{\alpha_n}\}$ is linearly independent over the algebraic numbers. So, in particular, $(E(M),+)$ is an FFM. Since $E(M) \cap (0,1) = \emptyset$, the semiring $E(M)$ is bi-reduced and, by \cite[Proposition 4.1]{NBSCFG}, bi-atomic. However, $(E(M)^{\bullet}, \cdot)$ is not an FFM. Indeed, the multiplicative submonoid $\{e^m \mid m \in M\}$ is isomorphic to $M$, which is obviously not an FFM. Therefore, $E(M)$ is not a bi-FFS.
\end{example}

\smallskip

 \section*{Acknowledgments}
 I am grateful to Felix Gotti for his guidance during the preparation of this paper, in particular, for many useful conversations that lead up to the discovery of Theorem~\ref{theorem: characterization of positive monoids that are FFMs}. While working on the same, I was generously supported by the University of Florida Mathematics Department Fellowship and the CAM Summer Research Fellowship.

\end{document}